\documentclass[reqno]{amsart}

\usepackage{epsf}
\usepackage{graphics}
\usepackage{graphicx}
\usepackage{amssymb}
\usepackage{amsmath}

\date{}

\theoremstyle{plain}
\newtheorem{theorem}{Theorem}
\newtheorem{corollary}{Corollary}
\newtheorem{proposition}{Proposition}

\theoremstyle{definition}

\newtheorem{remark}{Remark}

\theoremstyle{remark}

\newcommand\Z{{\mathbb Z}}

\usepackage{xcolor}

\title{Bouquets of curves in surfaces}

\author{S.~Baader, P.~Feller, L.~Ryffel}

\begin{document}

\begin{abstract} We characterise when a set of simple closed curves in an orientable surface forms a bouquet, in terms of relations between the corresponding Dehn twists.
\end{abstract}

\maketitle

\section{Introduction}

The \emph{mapping class group} of an oriented compact surface is the set of isotopy classes of its orientation-preserving self-diffeomorphisms with group law induced by composition. The mapping class group of an oriented closed surface is generated by Dehn twists along simple closed curves. This is due to the fact that a mapping class is essentially determined by its action on the set of simple closed curves; see Paragraph~10 in~\cite{De} and Chapter~4 in~\cite{Bi}. Dehn twists store a lot of information about curves; most importantly, the isotopy type of their defining curves.
 Two positive Dehn twists with non-isotopic defining curves detect low intersection numbers: they commute, or satisfy the braid relation, if and only if their defining curves have intersection number zero or one, respectively; see Chapter~3 in~\cite{FM}. Here two group elements $g$ and $h$ are said to satisfy the \emph{braid relation} if $ghg=hgh$.
 
A \emph{bouquet} in a surface is a union of $n$~simple closed curves that have precisely one common intersection point in which all curves intersect pairwise transversally (that is with~$n$ different tangent lines). We say a set of pairwise non-isotopic simple closed curves in a surface \emph{forms a bouquet}, if their union is a bouquet after an individual isotopy of the curves involved,
and a set of curves 
indexed by $\Z/n\Z$ \emph{forms
an oriented bouquet} if they form a bouquet such that the cyclic order of the tangent vectors at the
common intersection point agrees with the one induced by the standard cyclic order 
of $\Z/n\Z$.
In this note, we derive a group theoretic characterization of oriented bouquets in terms of Dehn twists.

\begin{theorem} \label{bouquet} Let $\Sigma$ be an oriented compact surface and $n\geq 2$ an integer.
A set of simple closed curves $c_1,c_2,\cdots, c_n$ in $\Sigma$ 
forms
an oriented bouquet if and only if the corresponding positive Dehn twists $T_1,T_2,\ldots,T_n$ are not all equal and satisfy the following relations:
\begin{enumerate}
\item [(i)] the braid relation $T_i T_j T_i=T_j T_i T_j$, for all pairs $i,j \in\Z/n\Z$, and
\item [(ii)] the cycle relation $T_i T_j T_k T_i=T_j T_k T_i T_j$, for all triples $i,j,k\in\Z/n\Z$ of pairwise distinct indices such that $j$ is after $i$ and before $k$ in the standard cyclic order.
\end{enumerate}
\end{theorem}
We briefly comment on the conditions~(i) and (ii).
According to the discussion above, condition~(i) is equivalent to the pairwise intersection numbers being one. For sets of three or more curves, forming a bouquet is a strictly stronger condition. For example, a triple of curves that pairwise intersect once transversally needs to delimit a triangle on the surface in order to form a bouquet; see Section~\ref{sec:tripbouq}.
By Theorem~\ref{bouquet} the cycle relation is the additional condition needed to characterise bouquets. While less prominent than the braid relation, the cycle relation features in several geometric contexts. For example, it appears in the work of L\"onne on the monodromy group of simple plane curve singularities~\cite{L} (see also~\cite{PS} for a recent description of that group as a framed mapping class group). It also plays an important role in the definition of mutation-invariant groups associated with Dynkin type quivers introduced by Grant and Marsh in~\cite{GM}.


The key observation on which Theorem~\ref{bouquet} relies is the following group theoretic fact, the first part of which is a reformulation of a consequence of a result by Birman and Hilden~\cite{BH}, while the second part is an algebraic consequence of Artin's standard braid group presentation.

\begin{proposition} \label{braidgroup}
Let $c_1,c_2,\ldots,c_n$ form a $\pi_1$-injective bouquet in an oriented compact surface~$\Sigma$. Then the subgroup of the mapping class group of $\Sigma$ generated by the corresponding positive Dehn twists $T_1,T_2,\ldots,T_n$ is isomorphic to the braid group $B_{n+1}$. Moreover, the braid and cycle relations (i) and (ii) form a complete set of relations for the generators $T_1,T_2,\ldots,T_n$.
\end{proposition}

A subset of $\Sigma$ is \emph{$\pi_1$-injective} if the canonical inclusion induces an injection of its fundamental group into $\pi_1(\Sigma)$. 
Dropping the assumption of $\pi_1$-injectivity, one would still 
get a quotient of
$B_{n+1}$ rather than $B_{n+1}$ itself.
We derive this proposition in the next section, since it is hard to extract from the existing literature. In the third section, we show that the cycle relation together with the braid relation characterises bouquets of $3$ curves. The generalisation from $3$ to $n$~curves is then purely topological, as we will see in Section~\ref{sec:genbouquets}. 

{After a first version of this article appeared as a preprint, it was pointed out to us that the generalisation from $3$ to $n$~curves was previously proven by Aougab and Gaster; see Proposition 5.3 in~\cite{AougabGaster_17}. In other words, once established for the case $n=3$ as is done in Section~\ref{sec:tripbouq}, the if statement of Theorem~\ref{bouquet} follows from their work. To be self-contained, we keep our arguments from Section~\ref{sec:genbouquets}. We hope that our complementary treatment will steer the reader to Aougab and Gaster's text, which has much to offer beyond Proposition~5.3.}

\subsection*{Acknowledgements} SB and PF gratefully acknowledge support by the SNSF Grant 178756 and
the SNSF Grant 181199, respectively.

\section{Bouquets and braid groups} \label{sec:braidgroups}

We denote by $T_a
$ the positive Dehn twist along a simple closed curve $a$ in an oriented surface $\Sigma$. Given two simple closed curves $a,b \subset \Sigma$ that intersect transversally in one point, we obtain the following equality between curves, up to isotopy: $T_a(b)=T_b^{-1}(a)$. Rewriting this as $T_b T_a(b)=a$ and applying the change of coordinates $T_{T_b T_a(b)}=(T_a T_b)T_a(T_a T_b)^{-1}$, we obtain that $T_a$ and $T_b$ satisfy the braid relation
\[T_a T_b T_a=T_b T_a T_b.\]
For a more detailed proof, including the reverse implication; see Chapter~3 in~\cite{FM}.
More generally, let $a_1,a_2,\ldots,a_n \subset \Sigma$ be a of set of simple closed curves such that their union is $\pi_1$-injective in $\Sigma$ and such that they are pairwise disjoint, except for pairs with consecutive indices, which intersect transversally in one point. Such a family of curves is called a chain.

For every oriented compact surface $\Sigma$, the subgroup of the mapping class group of  $\Sigma$
generated by the Dehn twists associated with a chain of $n$~curves is isomorphic to the braid group~$B_{n+1}$. This is a consequence of work by Birman and Hilden in~\cite{BH} (see also Chapter~9 in~\cite{FM}). An interpretation of that subgroup as the monodromy group of a plane curve singularity of type~$A_n$ was later described in~\cite{PV}; the case of curves intersecting in a general tree-like pattern was solved by Wajnryb in~\cite{W}. The $\pi_1$-injectivity is needed to rule out `false chains', such as $a,b,\bar{a}$, where the curves~$a$ and~$\bar{a}$ cobound an embedded annulus. In that case, the resulting subgroup is isomorphic to the braid group $B_3$ or its quotient $\textrm{SL}(2,\Z)$ rather than~$B_4$.

Here is an important relation between bouquets and chains of curves: suppose that the simple closed curves  $a_1,a_2,\ldots,a_n \subset \Sigma$ form a $\pi_1$-injective bouquet, numbered in the anticlockwise direction around the common intersection point. Then the set of transformed curves
$$a_1,T_{a_1}^{-1}(a_2),T_{a_2}^{-1}(a_3)\ldots,T_{a_{n-1}}^{-1}(a_n)$$
forms a chain, as shown in Figure~1 for $n=4$ (where the new curves are labeled 1',2',3',4'). Moreover, the Dehn twists along these new curves generate the same subgroup in the mapping class group of $\Sigma$
as the Dehn twists associated with the curves of the initial bouquet. This is another consequence of the equation
$$T_{T_x^{-1}(y)}=T_x^{-1} T_y T_x.$$
Using the result about chains from the last paragraph, we conclude that the Dehn twists associated with the curves of a bouquet generate a subgroup isomorphic to the braid group $B_{n+1}$.

\begin{figure}[htb]
\begin{center}
\raisebox{-0mm}{\includegraphics[scale=1.0]{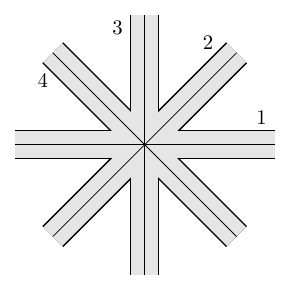}}
\qquad
\qquad
\raisebox{-0mm}{\includegraphics[scale=1.0]{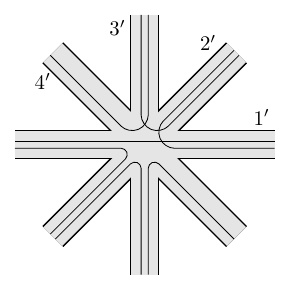}}
\caption{Bouquet and chain}\label{fig:bouquetandchain}
\end{center}
\end{figure}

As for the second statement of Proposition~\ref{braidgroup}, we need to analyse how the braid and cycle relations among the Dehn twists along the curves $a_1,a_2,\ldots,a_n$ translate into the usual braid and commutation relation among the Dehn twists associated with the transformed curves $a_1,T_{a_1}^{-1}(a_2),T_{a_2}^{-1}(a_3)\ldots,T_{a_{n-1}}^{-1}(a_n)$.

Let $a,b,c \in \{a_1,a_2,\ldots,a_n\}$ be a triple of curves ordered in the anticlockwise way, and let $x=a,y=T_a^{-1}(b),z=T_b^{-1}(c)$ be the transformed curves. The Dehn twists $T_x,T_y,T_z$ satisfy the two braid relations
$$T_x T_y T_x=T_y T_x T_y \ , \ T_y T_z T_y=T_z T_y T_z$$
and the commutation relation
$$T_x T_z=T_z T_x.$$
Moreover, this is a complete set of relations, again by the result about chains from above. 
We need to show that these are equivalent to the three braid relations
$$T_a T_b T_a=T_b T_a T_b \ , \ T_b T_c T_b=T_c T_b T_c \ , \ T_c T_a T_c=T_a T_c T_a$$
and the following version of the cycle relation, due to our choice of numbering:
$$T_b T_a T_c T_b=T_c T_b T_a T_c,$$
Deriving these relations from the braid relations among $T_x,T_y,T_z$ is an easy task, using the expressions
\begin{equation*}
\begin{split}
T_a &= T_x \\
T_b &= T_a T_y T_a^{-1}=T_x T_y T_x^{-1} \\
T_c &= T_b T_z T_b^{-1}=T_x T_y T_x^{-1} T_z T_x T_y^{-1} T_x^{-1}=T_x T_y T_z T_y^{-1} T_x^{-1}.
\end{split}
\end{equation*}

Indeed, after an identification of $T_x,T_y,T_z$ with the standard braid generators $\sigma_1,\sigma_2,\sigma_3 \in B_4$, the four relations among $T_a,T_b,T_c$ admit a pictorial proof:
$$T_a T_b T_a=\sigma_1^2 \sigma_2=T_b T_a T_b,$$
$$T_b T_c T_b=\sigma_1 \sigma_2^2 \sigma_3 \sigma_1^{-1}=T_c T_b T_c,$$
$$T_a T_c T_a=
{
\sigma_1^2 \sigma_2 \sigma_3 \sigma_2^{-1}
}
=T_c T_a T_c,$$
$$T_b T_a T_c T_b=\sigma_1^2 \sigma_2 \sigma_3=T_c T_b T_a T_c.$$


For the reverse direction, we express
\begin{equation*}
\begin{split}
T_x & = T_a \\
T_y &= T_a^{-1} T_b T_a \\
T_z &= T_b^{-1} T_c T_b, \\
\end{split}
\end{equation*}

\noindent
use the shortcuts $a=T_a,b=T_b,c=T_c$ in order to save space, and derive:
$$T_y T_x T_y=a^{-1}baaa^{-1}ba=baa=aa^{-1}baa=T_x T_y T_x.$$
Here we used the braid relation $bab=aba$. The second braid relation is a bit trickier:
\begin{equation*}
\begin{split}
T_z T_y T_z&=b^{-1}cba^{-1}bab^{-1}cb=cbc^{-1}a^{-1}bacbc^{-1}=cbc^{-1}a^{-1}cba,\\
T_y T_z T_y&=a^{-1}bab^{-1}cba^{-1}ba=a^{-1}bacbc^{-1}a^{-1}ba=a^{-1}cbba.
\end{split}
\end{equation*}
Here we used a version of the braid relation, $b^{-1}cb=cbc^{-1}$, as well as the cycle relation $bacb=cbac$. The equality $T_z T_y T_z=T_y T_z T_y$ is thus equivalent to
$$acbc^{-1}a^{-1}c=cb.$$
Thanks to the cycle relation $bacb=cbac$, the left hand side is equal to
$$b^{-1}cbacc^{-1}a^{-1}c=b^{-1}cbc=cb.$$
Finally, here is the commutation relation:
$$T_x T_z=ab^{-1}cb=acbc^{-1}=b^{-1}bacbc^{-1}=b^{-1}cbacc^{-1}=b^{-1}cba=T_z T_x.$$

A similar derivation of the equivalence of these two group presentations can be found in Section~2 of~\cite{BL}, where the cycle relation is used to define an invariant of positive braids. Applying the above procedure to all triples of curves among $a_1,a_2,\ldots,a_n$, we obtain a complete set of relations, as stated in Proposition~\ref{braidgroup}.

\section{Triple bouquets} \label{sec:tripbouq}

In this section we prove that whenever three simple closed curves $a,b,c$ in an oriented closed surface $\Sigma$ satisfy pairwise braid relations and a cycle relation, then the set of curves $a,b,c$ forms a bouquet or are all isotopic. Note that this settles Theorem~\ref{bouquet} for the case $n = 3$, since the converse follows from previous considerations. More concretely, in Section~\ref{sec:braidgroups} it was shown that the cycle relation follows algebraically from $T_xT_z = T_z T_x$, where $x = a$ and $z = T_b^{-1}(c)$.

\begin{figure}[htb]
  \centering
  \begin{minipage}{0.5\textwidth}
    \centering
    \includegraphics[scale=1.0]{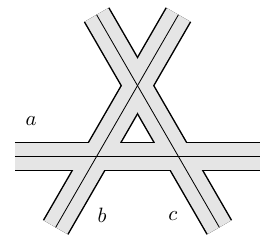}
  \end{minipage}%
  \begin{minipage}{0.5\textwidth}
    \centering
    \includegraphics[scale=1.0]{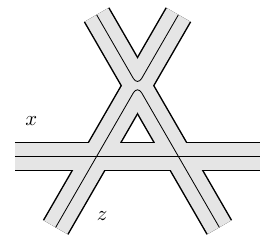}
  \end{minipage}

  \caption{The curves $x = a$ and $z = T_b^{-1}(c)$ intersect twice}
  \label{twist1}
\end{figure}

Suppose $a,b,c$ are curves satisfying the three braid relations
\begin{align*}
T_aT_bT_a &= T_bT_aT_b \\ T_aT_cT_a & =T_cT_aT_c \\ T_bT_cT_b& =T_cT_bT_c
\end{align*}
and the cycle relation  $T_b T_a T_c T_b = T_c T_b T_a T_c$. Using the relations, one checks that if two curves are isotopic, then $T_a=T_b=T_c$, so all 3 curves are isotopic. Thus, from here on, we consider $a$, $b$, and $c$ to be pairwise non-isotopic.
In particular, by the braid relations, $a,b,c$ have pairwise intersection number one. Hence, after an isotopy, they admit a tubular neighbourhood either as shown to the left of Figure~\ref{twist1}, in which case we write $a<b<c<a$, or mirrored, in which case we write $a < c < b < a$; compare Remark~\ref{rem:cyclordertriple} below.
Letting $x=a, z = T_b^{-1}(c)$ we have that $T_x$ and $T_z$ commute, under the exact same reasoning as in Section~\ref{sec:braidgroups}, where $T_xT_z = T_zT_x$ is deduced purely algebraically from the braid relations and the cycle relation.
This means that $x$ and $z$ in Figure~\ref{twist1} have disjoint representatives in their isotopy classes. Hence, $x$ and $z$ bound a bigon $B$, since their number of intersections is not minimal.

{
Notice that this allows us to exclude the case $a < c < b <a$, since
in this case, the curves $x$ and $z = T_b^{-1}(c)$
have intersection number two and hence $T_x$ and $T_z$ do not commute.
This is because the union $x \cup z$ bounds two regions,
both of whose boundaries are polygonal of length four
that intersect themselves in two corners,
and in particular the two regions are not bigons.
}

We can now assume that $a < b < c < a$. There are two possibilities for the position of $B$, indicated by the two dotted regions in Figure~\ref{bigons}. In the first case, on the left, it is obvious that $a,b,c$ form a bouquet. The second case, on the right, seems slightly more challenging. However, note that the two surfaces that are obtained by filling in the dotted regions are actually diffeomorphic via an orientation preserving diffeomorphism preserving all three curves $a,b,c$ individually as sets. One example of such a diffeomorphism is as follows. Cut up Figure~\ref{bigons} along the dashed lines,
as well as along the edges on the drawn boundary that are identified
with their opposites,
to obtain three $X$-shaped regions after six cuts. 
Rotating each of those 
regions
by 180 degrees preserves all identifications and maps the edges of the dotted triangle on the left to the edges of the dotted triangle on the right. Extending this to the dotted regions yields the desired diffeomorphism.

\begin{figure}[htb]
  \centering
  \begin{minipage}{0.5\textwidth}
    \centering
    \includegraphics[scale=1.0]{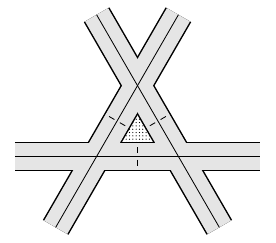}
  \end{minipage}%
  \begin{minipage}{0.5\textwidth}
    \centering
    \includegraphics[scale=1.0]{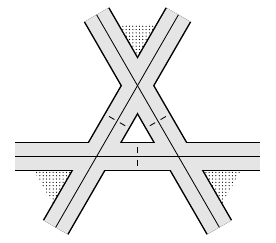}
  \end{minipage}

  \caption{Possible bigons}
  \label{bigons}
\end{figure}

It turns out that there are no further cases: the set $x \cup z$ 
bounds four regions in a small neighborhood of $x \cup z$, 
two of which we have now considered.
The other two cannot possibly be the boundary of a bigon,
since they have homotopically nontrivial boundary,
evidenced by the fact that the
boundary curves
intersect $b$ precisely once.

\section{General bouquets}\label{sec:genbouquets}
In this section, we prove Theorem~\ref{bouquet} by induction on the number of curves $n$. 
It is beneficial to be careful about the cyclic order of curves.
For a bouquet given as the union of $n$ simple closed curves $c_1,c_2,\ldots,c_n$ in an oriented surface $\Sigma$, we write
\[c_1<c_2<\cdots <c_n<c_1\] if $c_{i+1}$ occurs next (counter-clockwise) to $c_{i}$ for all $i\in\Z/n\Z$. For an example with $n=4$ and $c_1<c_2<c_3<c_4<c_1$; see the left-hand side of Figure~\ref{fig:bouquetandchain}. If a set of simple closed curves $c_1,c_2,\ldots,c_n$ forms a bouquet, we write $c_1<c_2<\cdots <c_n<c_1$ if the corresponding bouquet is as above. In other words, by the definition from the introduction, $c_1<c_2<\cdots <c_n<c_1$ means that $c_1,c_2,\ldots,c_n$ form an oriented bouquet.

\begin{remark}\label{rem:cyclordertriple} For a bouquet given as the union of three simple closed curves $a$, $b$, $c$ in $\Sigma$, we have $a<b<c<a$ if and only if isotoping $a$, $b$, $c$ into generic position (i.e., three distinct transversal intersection points realizing the pairwise intersection number one, respectively) yields that a regular neighborhood of $a\cup b\cup c$ is orientation-preservingly diffeomorphic to the one depicted on the left-hand-side of Figure~\ref{twist1}.

More generally, let $a,b,c$ be simple closed curves in $\Sigma$ that have pairwise intersection number one. Having 
$a<b<c<a$ and 
$a<c<b<a$, respectively, can be defined as in Section~\ref{sec:tripbouq}. And for bouquets of 3 curves the notions agree.
\end{remark}
Analyzing the case of 3 curves (as in Section~\ref{sec:tripbouq}) while keeping track of the cyclic order leads to the following proposition, which we use to prove Theorem~\ref{bouquet}. 
\begin{proposition}\label{prop:orderversionofmainresult} Fix $n\geq 2$.
Let $c_1,c_2,\ldots,c_n, c_{n+1}$ be simple closed curves in an oriented compact surface $\Sigma$ such that the set of $n$ curves $c_1,c_2,\ldots,c_n$ forms an oriented bouquet. Denote the positive Dehn twists along $c_{i}$ by~$T_i$.
If 
the $T_i$ 
satisfy
\begin{enumerate}
\item [(i')] the braid relation $T_i T_{n+1} T_i=T_{n+1} T_i T_{n+1}$ for all $1\leq i \leq n$ and
\item [(ii')] the cycle relation $T_{n} T_1 T_{n+1} T_{n} = T_{n+1} T_{n} T_1 T_{n+1}$ or one of its cyclic permutations,
\end{enumerate}
then the set of $n+1$ curves $c_1,c_2,\ldots,c_n, c_{n+1}$ forms an oriented bouquet.
\end{proposition}
As an aside, we note that $c_{n+1}$ being distinct from~$c_i$, up to isotopy, for $i\leq n$ is implied without being assumed.

\begin{proof}[Proof of Proposition~\ref{prop:orderversionofmainresult}] As a consequence of the bigon criterion, we can and do isotope all the $c_i$ to achieve that they intersect pairwise transversely and the following holds. The $c_1,c_2,\ldots,c_n$ intersect in the same point $p$
, and $c_i$ and $c_{n+1}$ realize their intersection number and are in general position (their intersections are pairwise different and different from $p$) for all $i\leq n$; see Figure~\ref{fig:n+1}~(A).
\begin{figure}[htb]
  \centering
  \begin{minipage}{0.35\textwidth}
       \centering
       \includegraphics{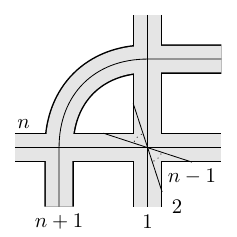}\\
       (A)
  \end{minipage}%
  \begin{minipage}{0.35\textwidth}
       \centering
       \includegraphics{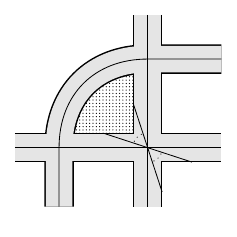}\\
       (B)
  \end{minipage}%
  \begin{minipage}{0.3\textwidth}
       \centering
       \includegraphics{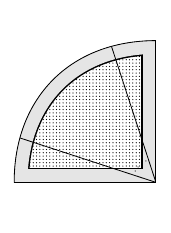}\\
       (C)
  \end{minipage}

  \caption{(A): A neighborhood of $c_1\cup c_n\cup c_{n+1}$ (grey), for $2\leq i\leq n-1$ the intersections between $c_i$ and $c_{n+1}$ are not drawn.
  \newline(B): That neighborhood union the triangle $\Delta$ (dotted).
  \newline(C): The region $C$ and its intersection with the $c_i$.}
  \label{fig:n+1}
  \end{figure}
We note that, due to (ii'), the curves $c_1$, $c_n$, and $c_{n+1}$ do intersect as depicted in Figure~\ref{fig:n+1} (A), rather than with the opposite cyclic order; see analysis of the cyclic order at the end of Section~\ref{sec:tripbouq}. We also note that $c_i$ and $c_{n+1}$ intersect at most once since they satisfy the braid relation~(i').

By the argument in Section~\ref{sec:tripbouq}, (i') and (ii') imply that the triple of curves $a=c_1$, $b=c_{n}$, and $c=c_{n+1}$ forms a bouquet. More precisely, up to an orientation preserving diffeomorphism, we have that a regular neighborhood of $a\cup b\cup c$ union a triangle $\Delta$ is embedded in $\Sigma$ as depicted in Figure~\ref{fig:n+1}~(B).

Denote by $C$ the connected component of $\Sigma\setminus (a \cup b \cup c)$ containing~$\Delta$. By the assumption on $c_1,c_2,\ldots,c_n$ forming an oriented bouquet, the triangle $C$ has nonempty intersections with all $c_i$ for $1\leq i\leq n$.
Hence, each $c_i$ intersects $C$ in an interval with its endpoints on $\partial C$: one at $p$ and the other one in the interior of the interval $c_{n+1}\cap \partial C$; see Figure~\ref{fig:n+1}~(C). Thus, after isotoping $c_{n+1}$ across $C$, we conclude that
$c_1,c_2,\ldots,c_n, c_{n+1}$ form an oriented bouquet. 
\end{proof}

\begin{proof}[Proof of Theorem~\ref{bouquet}]
For the case of $n=2$, recall from the first paragraph of the introduction that two non-isotopic simple closed curves can be isotoped to intersect once transversally if and only if the corresponding Dehn-twists satisfy the braid relation. For $n\geq 3$, we induct on $n$.
The base case (3 curves) was treated in Section~\ref{sec:tripbouq}. For the induction step, we assume as the induction hyphothesis that Theorem~\ref{bouquet} holds for a fixed $n\geq3$.

The only if statement follows from Proposition~\ref{braidgroup}.
For the if statement, consider $c_1,c_2,c_n,c_{n+1}$ in $\Sigma$ with corresponding positive Dehn twists $T_i$ along them satisfying~(i) and (ii) and not all $T_i$ are equal. By cyclically relabelling the curves if needed, we may and do assume that not all $T_1, T_2,\cdots,T_n$ are equal.
By the induction hypothesis, $c_1,c_2,\ldots,c_n$ form an oriented bouquet, i.e.~$c_1<c_2\cdots<c_n<c_1$.

We consider a 
{pair
of consecutive curves $b,a$}
in this bouquet; that means, $a=c_{i+1}$ and $b=c_{i}$ for $1\leq i \leq n-1$ or $a=c_1$ and $b=c_n$.
There is at least one choice of 
{$b,a$}
such that the cyclic order of $a,b,c_{n+1}$ (as defined in Remark~\ref{rem:cyclordertriple}) is $a<b<c_{n+1}<a$.
Indeed, assume we have $c_{i+1}<c_{n+1}<c_i<c_{i+1}$ for all $1\leq i\leq n-1$,
then one checks (using $c_1<c_2\ldots<c_n$) that $c_1<c_{n}<c_{n+1}<c_1$.

To conclude, we cyclically relabel  $c_1,c_2,\ldots,c_n$ such that $c_1<c_{n}<c_{n+1}<c_1$. Hence, by Remark~\ref{rem:cyclordertriple} the cycle relation for $c_1$, $c_n$, and $c_{n+1}$ provided by (ii) is \[T_{n} T_1 T_{n+1} T_{n} = T_{n+1} T_{n} T_1 T_{n+1}\] or one of its cyclic permutations. Thus, $c_1,c_2,\ldots,c_n, c_{n+1}$ form an oriented bouquet by Proposition~\ref{prop:orderversionofmainresult}. This concludes the induction step.
\end{proof}


\section{An explicit criterion}
From the proof of Theorem~\ref{bouquet}, one notices that we did not use all cycle relations as provided by the assumption (ii). Only linearly many cycle relations (in terms of number of curves) are needed.
Indeed, inductive application of Proposition~\ref{prop:orderversionofmainresult} yields the following. 
\begin{corollary}
Fix $n\geq 3$. Let $c_1,c_2,\ldots,c_n$ be simple closed curves in an oriented compact surface $\Sigma$  at least two of which are non-isotopic.
Denote the positive Dehn twists along $c_{i}$ by $T_i$.
Then, the set of $n$ curves $c_1,c_2,\ldots,c_n$ forms an oriented bouquet
if and only if 
the $T_i$ satisfy
\begin{enumerate}
\item [(i'')] the braid relation $T_i T_j T_i=T_{j} T_i T_{j}$ for all $1\leq i<j\leq n$ and
\item [(ii'')] the cycle relation $T_{i} T_1 T_{i+1} T_{i} = T_{i+1} T_{i} T_1 T_{i+1}$ or one of its cyclic permutations for all $2\leq i\leq n-1$.\qed
\end{enumerate}
\end{corollary}

\bigskip
\noindent
Universit\"at Bern, Sidlerstrasse 5, CH-3012 Bern, Switzerland

\smallskip
\noindent
ETH Z\"urich, R\"amistrasse 101, CH-8092 Z\"urich, Switzerland

\smallskip
\noindent
Universit\"at Bern, Sidlerstrasse 5, CH-3012 Bern, Switzerland

\bigskip
\noindent
\texttt{sebastian.baader@math.unibe.ch}

\smallskip
\noindent
\texttt{peter.feller@math.ch}

\smallskip
\noindent
\texttt{levi.ryffel@math.unibe.ch}

\end{document}